\documentclass[12pt,leqno,fleqn]{amsart}  
\usepackage{amsmath,amstext,amsthm,amssymb,amsxtra}
\usepackage[top=1.5in, bottom=1.5in, left=1.25in, right=1.25in]	{geometry}
\usepackage[normalem]{ulem}
\usepackage{txfonts} % pxfonts txfonts 
\usepackage[T1]{fontenc}
\usepackage{lmodern}
\usepackage{tikz}\usetikzlibrary{arrows}

\usepackage{euler}   % better than the option below

\makeatletter
\DeclareRobustCommand\widecheck[1]{{\mathpalette\@widecheck{#1}}}
\def\@widecheck#1#2{%
    \setbox\z@\hbox{\m@th$#1#2$}%
    \setbox\tw@\hbox{\m@th$#1%
       \widehat{%
          \vrule\@width\z@\@height\ht\z@
          \vrule\@height\z@\@width\wd\z@}$}%
    \dp\tw@-\ht\z@
    \@tempdima\ht\z@ \advance\@tempdima2\ht\tw@ \divide\@tempdima\thr@@
    \setbox\tw@\hbox{%
       \raise\@tempdima\hbox{\scalebox{1}[-1]{\lower\@tempdima\box
\tw@}}}%
    {\ooalign{\box\tw@ \cr \box\z@}}}
\makeatother
\usepackage{mathtools}
\mathtoolsset{showonlyrefs,showmanualtags}

\usepackage{hyperref} %,hypertexnames=false,colorlinks,[pagebackref]
\hypersetup{
    colorlinks=true,       % false: boxed links; true: colored links
    linkcolor=blue,          % color of internal links
    citecolor=magenta,        % color of links to bibliography
    filecolor=magenta,      % color of file links
    urlcolor=cyan           % color of external links
}

\usepackage[msc-links]{amsrefs}

\theoremstyle{plain} % definition 
\newtheorem{lemma}[equation]{Lemma} 
\newtheorem{proposition}[equation]{Proposition} 
\newtheorem{theorem}[equation]{Theorem}

\theoremstyle{definition}

\theoremstyle{remark}

\numberwithin{equation}{section}

\title  {Lacunary  Discrete Spherical Maximal Functions}
\author[R. Kesler]{Robert Kesler}
\address{ School of Mathematics, Georgia Institute of Technology, Atlanta GA 30332, USA}
\email{robertmkesler@gmail.com}

\author[M. T. Lacey] {Michael T. Lacey}   %  can use \and  

\address{ School of Mathematics, Georgia Institute of Technology, Atlanta GA 30332, USA}
\email {lacey@math.gatech.edu}
\thanks{Research supported in part by grant  from the US National Science Foundation, DMS-1600693 and the 
Australian Research Council ARC DP160100153.}

 \author[D. Mena]{Dar\'io Mena Arias} 
\address{Universidad de Costa Rica}
\email {DARIO.MENAARIAS@ucr.ac.cr}

\begin{document}
\begin{abstract}
We  prove new  $\ell ^{p} (\mathbb Z ^{d})$ bounds for discrete spherical averages in 
dimensions $ d \geq 5$.  
We focus on the case of lacunary radii, first for general lacunary radii, and then for certain kinds of highly composite choices of radii.  
In particular, if $ A _{\lambda } f $ is the spherical average of $ f$ over the discrete sphere of radius $ \lambda $, we have 
\begin{equation*}
\bigl\lVert  \sup _{k}  \lvert  A _{\lambda _k} f \rvert \bigr\rVert _{\ell ^{p} (\mathbb Z ^{d})} 
\lesssim \lVert f\rVert _{\ell ^{p} (\mathbb Z ^{d})}, \qquad   \tfrac{d-2} {d-3} < p \leq \tfrac{d} {d-2},\ d\geq 5,  
\end{equation*}
for any lacunary sets of integers $ \{\lambda _k ^2 \}$.  
We follow a style of argument from our prior paper, addressing the full supremum. The relevant maximal operator is decomposed into several parts; each part requires only one endpoint estimate.  
\end{abstract}

	\maketitle  
\tableofcontents 
%%%%%%%%%%%%%%%%%%%%%%%%%%%%%% SECTION  SECTION SECTION
%%%%%%%%%%%%%%%%%%%%%%%%%%%%%% SECTION  SECTION SECTION 
\section{Introduction} %\label{s:}  

We prove $ \ell ^{p}$ bounds for discrete spherical maximal operators, concentrating on variants of the lacunary versions of these operators.  
They have a surprising intricacy.  
 For $ \lambda^2  \in \mathbb N $, let $  s_{\lambda } $ be the cardinality of 
the number of $ n\in \mathbb Z ^{d}$ such that $ \lvert  n\rvert ^2   = \lambda  ^2 $.   Define the spherical 
average of a function $ f $ on $ \mathbb Z ^{d}$ to be 
\begin{equation*}
A _{\lambda } f (x) =  s _{\lambda } ^{-1} \sum_{n \in \mathbb Z ^{d } \;:\; \lvert  n\rvert^2 = \lambda  } 
f (x-n)
\end{equation*}
We will always work in dimension $ d \geq 5$, so that for any choice of \( \lambda ^{2} \in \mathbb N  \), 
one has $ s _{\lambda } \simeq \lambda ^{d-2}$.
Define the maximal function $ A _{\ast } f = \sup _{\lambda } A _{\lambda } f $, 
where $ f $ is non-negative and the supremum is over all $ \lambda $ for which the operator is defined. 
This operator was introduced by Magyar \cite{MR1617657}, and the $ \ell ^{p}$ bounds were proved by Magyar, Stein and Wainger \cite{MSW}.   Namely, this is a bounded operator on $ \ell ^{p}$ for 
$ p > \frac{d} {d-2}$.  
 
We address the discrete lacunary spherical maximal function.  
We say that a set of integers \( \{\lambda _k ^{2} \;:\; k\geq 1\} \) is \emph{lacunary} if \( \lambda _{k+1} ^{2}\geq 2 \lambda _k ^{2}\) for all \( k\in \mathbb N \).   
Let \(  A _{ \textup{lac} } = \sup _{k\in \mathbb Z } A _{ \lambda _k} f\).  
We will see that the choice of the \( \lambda _{k} \) have a strong impact on the results.

%%%%%%%%%%%%%%%%%%%%%%%%%%%%%%
\begin{theorem}   \label{t:lac}  For \( d\geq 5 \), let   \( \{\lambda _k  ^2 \}  \) be any lacunary sequence of integers.  
The maximal operator \(  A _{ \textup{lac} }  \) maps $ \ell ^{p}(\mathbb Z ^{d}) \to \ell ^{p}(\mathbb Z ^{d})$ for $ p > \frac{d-2} {d-3}$. 
\end{theorem}
%%%%%%%%%%%%%%%%%%%%%%%%%%%%%%
 
Our bound $\frac{d-2} {d-3} $ is smaller than the index $ \frac{d} {d-2}$, for which the 
full supremum $  A _{\ast} f$ is bounded \cite{MSW}.  
Kevin Hughes \cite{160904313} proved a version of the result above, for a very particular sequence of radii, and in dimension $ d=4$.  
In contrast to the continuous case, no such inequalities can hold close to $ \ell ^{1}$. 
An example of Zienkiewicz \cite{JZ}
show that there are lacunary radii \( \{ \lambda_k\} \) for which the corresponding maximal operator \( A _{ 
\textup{lac}}  \) is unbounded on $ \ell ^{p}$, for $ 1< p < \frac {d} {d-1}$. 
It is an interesting question  to determine the best $ p = p (d)$ for which  any lacunary 
maximal function $ A _{\textup{lac}} $ would be bounded on $ \ell ^{p} (\mathbb Z ^{d})$.

\bigskip 

The Theorem above concerns classical type examples of radii.
Brian Cook \cite{180803822} has shown that for   highly composite  radii 
$ \lambda _{k} ^2 = 2 ^{k}!$, that the maximal function $   \sup _{k} A _{\lambda _k} f$ is 
bounded on $ \ell ^{p}$, for all $ 1< p < \infty $.   
The Theorem below shows that this continues to hold for e.g.~$ \lambda _k ^2 = [ k ^{\log\log k}]!$.  

%%%%%%%%%%%%%%%%%%%%%%%%%%%%%% THEOREM THEOREM THEOREM
\begin{theorem}\label{t:HC}  For \( d\geq 5 \), let $ \mu _k $ be an increasing  sequence of integers for which 
\begin{equation}\label{e:log}
\lim _{k} \frac{\log \mu _k} {\log k} = \infty . 
\end{equation}
Then, for $ \lambda _k ^2 = \mu _k !$, the maximal function $ \sup _{k} A _{\lambda _k} f$ 
maps $ \ell ^{p} (\mathbb Z ^{d})\to \ell ^{p}(\mathbb Z ^{d})$ for $ 1< p < \infty $.   
\end{theorem}
%%%%%%%%%%%%%%%%%%%%%%%%%%%%%% THEOREM THEOREM THEOREM

\bigskip 
Our  method of proof is inspired by a method of Bourgain \cite{MR812567}, and its application to the 
discrete setting by Ionescu \cite{I}.  We used it for the full discrete spherical maximal operator of Magyar, Stein and 
Wainger in \cite{MR4041278}.   In particular, we proved an endpoint sparse bound in that setting. 
 
These arguments are relatively easy.  
The maximal operators are treated as maximal multipliers.  
Each component of the decomposition of the multiplier needs only one 
estimate, either an $ \ell ^2 $ estimate, or  an $ \ell ^{1}$ estimate.  As such, the argument can be used to simplify existing results, and simplify the search for new ones.   
We illustrate these ideas in a simple context in \S \ref{s:R}.
The discrete lacunary theorem is proved in \S\ref{s:lac}, and the highly composite case in \S~\ref{s:HC}. 

%%%%%%%%%%%%%%%%%%%%%%%%%%%%%% SECTION  SECTION SECTION
%%%%%%%%%%%%%%%%%%%%%%%%%%%%%% SECTION  SECTION SECTION 
\section{The Continuous Lacunary Case} \label{s:R}

To illustrate the proof technique, we prove the classical results on the lacunary spherical averages on Euclidean spaces. 
Let $ \mathbb S ^{d-1}$ be the unit sphere in $ \mathbb R ^{d}$, and let $ \sigma $ be the rotationally invariant probability measure on $ \mathbb S ^{d-1}$.  Let 
\begin{equation*}
\mathcal A _{\lambda } f (x) = \int _{\mathbb S } f (x-y) \; d \sigma (y). 
\end{equation*}
The key property of these averages that we will rely upon is the stationary decay estimate 
\begin{equation}\label{e:STA}
\lvert  \widetilde {d \sigma } (\xi )\rvert  \lesssim \lvert  \xi \rvert  ^{- \frac{d-1}2}, 
\end{equation}
where the tilde represents the Fourier transform.  
We begin with this proposition. 

%%%%%%%%%%%%%%%%%%%%%%%%%%%%%% PROPOSITION PROPOSITION PROPOSITION
\begin{proposition}\label{p:Littman} For $ f = \mathbf 1_{F}$ and $ g= \mathbf 1_{G}$ supported on the unit cube in $ \mathbb R ^{d}$, 
there holds 
\begin{equation*}
\langle \mathcal A _1   \mathbf 1_{F} , \mathbf 1_{G}\rangle \lesssim (\lvert  F\rvert \cdot \lvert  G\rvert  ) ^{\frac{d} {d+1}}, 
\qquad  F,G \subset [0,1] ^{d}. 
\end{equation*}

\end{proposition}
%%%%%%%%%%%%%%%%%%%%%%%%%%%%%% PROPOSITION PROPOSITION PROPOSITION

The inequality above is just a little weaker than the classical result of Littman \cite{MR0358443} and Strichartz \cite{MR0256219}, 
that locally $ \mathcal A_1$ maps $ L ^{ \frac{d+1}d}$ into $ L ^{d+1}$.  That inequality requires a sophisticated analytic interpolation argument. 

%%%%%%%%%%%%%%%%%%%%%%%%%%%%%% PROOF PROOF PROOF
\begin{proof}
The proof proceeds by this supplementary procedure.  For integers $ N$, we estimate $ \mathcal A_1 f \leq M_1 + M_2$, where 
\begin{equation}\label{e:L1}
\lVert M_1\rVert_ \infty  \leq N \lvert  F\rvert, \qquad \lVert M_2\rVert _2 \leq N ^{- \frac{d-1}2}\lvert  F\rvert ^{1/2}.  
\end{equation}
With this established, we have 
\begin{equation*}
\langle \mathcal A _1   \mathbf 1_{F} , \mathbf 1_{G}\rangle 
\leq  N \lvert  F\rvert \cdot \lvert  G\rvert +    N ^{- \frac{d-1}2}\bigl[ \lvert  F\rvert \cdot \lvert  G\rvert \bigr] ^{1/2}
\end{equation*}
Optimizing the right hand side over $ N$ proves the proposition. We omit the details.  

It remains to construct $ M_1$ and $ M_2$. 
 Let $\varphi $ be a non-negative  Schwartz function, with integral one, and compact spatial  support. 
Likewise, set $ \varphi _ t (x) = t ^{-d} \varphi (x/t)$.  Then, $ M_1 = \varphi _{1/N} \ast \mathcal A_1 f $. 
This is convolution of $ f$ against a uniform probability measure supported on an annulus around the unit sphere of width $ 1/N$. 
So it is clear that $ M_1$ satisfies the first estimate in \eqref{e:L1}, and the second estimate  \eqref{e:L1} for $ M_2$ follows from \eqref{e:STA}. 
(This proof is known to experts in the subject.) 
\end{proof}
%%%%%%%%%%%%%%%%%%%%%%%%%%%%%% PROOF PROOF PROOF

The next argument addresses the lacunary spherical maximal function.

%%%%%%%%%%%%%%%%%%%%%%%%%%%%%% THEOREM THEOREM THEOREM
\begin{theorem}\label{t:R}  Let $ \{\lambda _k\} \subset (0, \infty )$ be a lacunary sequence of reals. 
Then, there holds 
\begin{equation} \label{e:RL}
\lVert \sup _{k} \mathcal A _{\lambda _k} f \rVert _{p} \lesssim \lVert f\rVert_p, \qquad 1< p < \infty . 
\end{equation}
\end{theorem}
%%%%%%%%%%%%%%%%%%%%%%%%%%%%%% THEOREM THEOREM THEOREM

%%%%%%%%%%%%%%%%%%%%%%%%%%%%%% PROOF PROOF PROOF
\begin{proof}
The inequality in \eqref{e:RL} is elementary for $ p=2$. And we take it for granted, while noting that a certain quantification of this familiar argument will appear below.  
It remains to prove the inequality for $ 1< p < 2$. 
We aim to prove the restricted weak type estimate 
\begin{equation}\label{e:RWTE}
\bigl\langle \sup _{k} \mathcal A _{\lambda _k} f, g  \bigr\rangle \lesssim \lvert  F\rvert ^{1/p} \lvert  G\rvert ^{1/p'},   
\end{equation}
where $ f = \mathbf 1_{F}$ and $ g= \mathbf 1_{G}$.  Note that the $ L ^{2} $ inequality implies this for $ \lvert  G\rvert \leq \lvert  F\rvert $. So we assume 
the converse below.  

We set up a supplementary objective.  For sets $ F\subset \mathbb R ^{d}$ of finite measure, choices of $1< p < 2 $, and all integers $ N$, 
we can write $  \sup _{k} \mathcal A _{\lambda _k} f \leq M_1 + M_2$, 
\begin{align}\label{e:R1}
\lVert M_1 \rVert  & \lesssim  (\log N)\lvert  F\rvert ^{1/ p} , 
\\ \label{e:RRR2}
\lVert M_2 \rVert _{2} & \lesssim N ^{ - \frac{d-1}2 } \lvert  F\rvert ^{1/2}.  
\end{align}

We have 
\begin{align*}
\bigl\langle \sup _{k} \mathcal A _{\lambda _k} f, g  \bigr\rangle  & \lesssim  
\langle M_1 , \mathbf 1_{G} \rangle + \langle M_2 , \mathbf 1_{G} \rangle 
\\
& \lesssim (\log N) \lvert  F\rvert ^{1/ p} \lvert  G\rvert ^{ (p-1)/p } 
+ N ^{ - \frac{d-1}2 } \lvert  F\rvert ^{1/2} \lvert  G\rvert ^{1/2}  .
\end{align*}
Recalling that $ \lvert  G\rvert > \lvert  F\rvert  $, we can optimize this over $ N$, and then let $ p $ tend  to one to complete the proof of \eqref{e:RWTE}. We omit the details, except to say that the restriction to indicators is very useful at this point.  

\smallskip 

We turn to the construction of $ M_1 $ and $ M_2$.  Using the same notation is in the proof of Proposition~\ref{p:Littman}, set 
\begin{equation*}
M_1 = \sup _{k}  \varphi _{ \lambda _k /N} \ast \mathcal A _{\lambda _k} f. 
\end{equation*}
This defines $ M_2$ implicitly.  The   stationary decay estimate \eqref{e:STA} and a standard square function argument  
combine in a familiar way to prove \eqref{e:RRR2}. 
\begin{align*}
\lVert M_2 \rVert _ 2 ^2  & \lesssim  \sum_{k} \lVert  \varphi _{ \lambda _k /N} \ast \mathcal A _{\lambda _k} f - \mathcal A _{\lambda _k } f   \rVert_2 ^2 
\\
& \lesssim \lVert f\rVert_2 ^2 \sup _{\xi } \sum_{k} \lvert  \widetilde  \varphi (\lambda _k  \xi ) -1   \rvert ^2 \cdot \lvert  \widetilde {d \sigma } (\xi ) \rvert ^2 
\\
& \lesssim N ^{1-d} \lvert  F\rvert .   
\end{align*}
Note that this argument is a certain quantification of the standard square function proof of the boundedness of the lacunary spherical maximal operator on $ L ^2 $.  

\smallskip 
For \eqref{e:R1}, namely the control of $ M_1$, we show that the maximal function $ B_N f = \sup _{k} \varphi _{ \lambda _k /N} \ast \mathcal A _{\lambda _k} f$ 
satisfies a strong type $ L ^{p}$ bound smaller than $ \log  N$. 

Now, it is clear that $ B_N$ is a bounded operator on $ L ^2 $.  
One can approach the $ L ^{p}$ bounds for $ 1< p < 2$ directly, using a bit of Calder\'on-Zygmund theory. 
We use duality, however. 
This requires that we linearize the maximal operator $B_N f$, which is done as follows.

For any collection of pairwise disjoint subsets of $ \mathbb R ^{d}$ denoted by 
$ \{  S_k \;:\; k\in \mathbb Z \}$, we can form the linear operator 
\begin{equation*}
T f = \sum_{k} \mathbf 1_{S_k} \varphi _{\lambda _k/N} \ast \mathcal A _{\lambda _k } f . 
\end{equation*}
This is bounded on $ L ^2 $, independently of the selection of the sets $ S_k$.  
We show that $ T ^{\ast} $ maps $ L ^{\infty }$ into $ BMO$ with norm at most $ \log N$.  
By interpolation and duality, we see that \eqref{e:R1} holds. 

 To verify our $ BMO$ claim we need to show  this: 
For $ \phi  \in L ^{\infty }$, 
and cube $ Q $, there is a constant $ \mu $ so that 
\begin{equation}\label{e:LBMO} 
\int _{Q} \lvert  T ^{\ast} \phi   - \mu \rvert ^2  \lesssim   (\log N) ^2 \lVert \phi \rVert_ \infty ^2 \lvert  Q\rvert . 
\end{equation}
Split $ T^\ast $ into three parts, $ T^\ast_0, T^\ast_1, T^\ast_2$, where 
\begin{align*}
T^\ast_0 \phi  &= \sum_{k \;:\; \lambda _k  <  \ell  Q} 
\varphi _{\lambda _k/N} \ast \mathcal A _{\lambda _k } ( \mathbf 1_{S_k} \phi )  , 
\\
T^\ast_2 \phi  &= \sum_{k \;:\;     \ell  Q  < \lambda _k/N}   
\varphi _{\lambda _k/N} \ast \mathcal A _{\lambda _k } ( \mathbf 1_{S_k} \phi )  , 
\end{align*}
This defines $ T^\ast_1$ implicitly.  
Define $ \mu = T^\ast_2 \phi  (x_Q)$, where $ x_Q$ is the center of $ Q$.  
Straight forward kernel estimates and lacunarity of $ \lambda _k$ show that 
\begin{equation*}
\sup _{x\in Q}\lvert  T^\ast _{2} \phi  (x) - \mu  \rvert \lesssim \lVert \phi \rVert_ \infty .  
\end{equation*}
For $ T^\ast_0$, we have the $ L ^2 $ bound for $ T^\ast$ which implies 
\begin{align*}
\int _{Q} \lvert  T^\ast _0 \phi   \rvert ^2 \; dx & = \int _{Q} \lvert  T^\ast _0( \phi  \mathbf 1_{2Q})  \rvert ^2 \; dx 
\lesssim \lVert \phi \rVert_ \infty ^2 \lvert  Q\rvert . 
\end{align*}
That  leaves $ T^\ast_1$, but it is the sum of at most $ \log N$ functions each bounded by $ \lVert \phi \rVert _{\infty }$.  Thus, \eqref{e:LBMO} follows.

\end{proof}
%%%%%%%%%%%%%%%%%%%%%%%%%%%%%% PROOF PROOF PROOF

We make these additional remarks on this method of proof used in this paper. 
%%  ENUMERATE
\begin{enumerate}

\item The fine analysis of the $ L ^{1}$ endpoint of the continuous lacunary spherical maximal function is still an open question \cites{MR1955209,2017arXiv170301508C}. 
It would be interesting to know if this technique can simplify those arguments. 

\item  For the local maximal operator $ \sup _{1\leq \lambda \leq 2} \mathcal A _{\lambda } f$, considered by Schlag \cite{MR1388870}, 
there is an elegant proof of the $ L ^{p}$ improving estimates along these lines of this section, given by Sanghyuk Lee \cite{MR1949873}.  
The latter argument can be modified in an interesting way to prove sparse variants for the Stein maximal operator, giving certain 
improvements over the sparse bounds of \cite{MR4041115}. 

\item Likewise, the $ \ell ^{1}$ endpoint cases are of interest in the discrete case. 
Can one show that for the maximal functions $ M$ in Theorem~\ref{t:HC}, that they map $ \ell \log \ell $ into weak $ \ell ^{1}$? 

\item The two proofs can be combined to prove a restricted weak type sparse bound for the lacunary spherical maximal function at the 
point $ (\frac{d+1}d, \frac{d+1}d)$. This is an interesting extension of the sparse bounds proved in \cite{MR4041115}. We leave the details to the reader.  

\item The main results of \cite{MR4041278} prove sparse bounds for the Magyar Stein Wainger discrete spherical maximal function.  
Those inequalities can be combined with Theorem~\ref{t:lac} and Theorem~\ref{t:HC} to give novel sparse bounds for 
these operators. These in turn imply novel weighted inequalities, which we leave to the interested reader.
However, in the special case of Theorem~\ref{t:lac}, one can prove additional sparse bounds. We do not purse these details here. 
\end{enumerate}
%% ENUMERATE
 
We thank the referee for encouraging us  to include this section in the paper. 
 
%%%%%%%%%%%%%%%%%%%%%%%%%%%%%% SECTION  SECTION SECTION
%%%%%%%%%%%%%%%%%%%%%%%%%%%%%% SECTION  SECTION SECTION 
\section{General Lacunary Sequences} \label{s:lac}

 The key Lemma is the restricted  type estimate below.

\begin{lemma}\label{l:lacLemma} Let $ \lambda _k ^2 $ be a lacunary set of integers.  
 For a finitely supported function \( f = \mathbf{1}_F \), and function $ \tau \;:\; \mathbb Z ^{d} \to \{  \lambda _k\}$, there holds 
 \begin{align}  \label{e:lacPre}
\lVert  A _{ \tau} f\rVert_ p \lesssim  \lvert  F\rvert ^{1/p}  , \qquad  \tfrac{d-2} {d-3} < p < 2. 
 \end{align}
\end{lemma}

We will use the stopping time $ \tau $ to simplify notation throughout.  
We turn to the proof.   It suffices to show that for all integers $ N$, we can decompose $ A _{\tau } f  \leq M_1 + M_2$ with 
\begin{equation}\label{e:Ms}
\lVert M_1 \rVert_{1+ \epsilon } \lesssim N ^{1+ \epsilon } \lVert f\rVert _{1+ \epsilon } , 
\qquad  
\lVert M_2 \rVert_{2} \lesssim  N ^{- \frac{4-d}2}\lVert f\rVert _{2} . 
\end{equation}
Above, implied constants depend upon $ 0< \epsilon < 1$, but we do not make this explicit here, nor at any point of the paper.  
Optimizing over $ N$ proves   \eqref{e:lacPre}. 

Both $ M _{1}$ and $ M _{2}$ have several parts. The first part of $ M _{1}$ is 
$ M _{1,1} =  \mathbf 1_{\tau \leq N}  A _{\lambda _k }f$.  It trivially satisfies the first half of \eqref{e:Ms}.

Recall the decomposition of $ A _{\lambda } f $ from Magyar, Stein and Wainger \cite{MSW}. 
We have the decomposition below, in which  upper case letters denote a convolution operator, and lower case letters denote the corresponding multiplier.  Let $ e (x) = e ^{2 \pi i x}$ and for integers $ q$, $ e_q (x) = e (x/q)$.   
\begin{align}  \label{e:Afull}
A _{\lambda } f &= C _{\lambda } f + E_{\lambda } f , 
\\ \label{e:Cfull}
C _{\lambda } f &= \sum_{1\leq \lambda \leq q} \sum_{ a \in \mathbb Z _q ^{\times }}  e_q (- \lambda ^2 a)  C ^{a/q} _{\lambda } f ,
\\  \label{e:caq}
c ^{a/q} _{\lambda }(\xi) &= \widehat {C ^{a/q} _{\lambda }}   (\xi ) = \sum_{\ell \in \mathbb Z ^{ d}_q} G(a, \ell, q) \widetilde \psi_q (\xi - \ell /q) 
\widetilde {d \sigma _{\lambda }} (\xi - \ell /q) 
\\  \label{e:Gauss}
G(a, \ell, q) &= q ^{-d}\sum_{n \in \mathbb Z _q ^{d}} e_q ( \lvert  n\rvert ^2 a+n \cdot \ell  ). 
\end{align}
The term $ G(a, \ell, q)$ is a normalized Gauss sum. Above,  $ a $ is in the multiplicative 
group $ \mathbb Z _q ^{\times }$. 
Recall that 
\begin{equation}\label{e:G<}
\lvert  G (a, \ell ,q)\rvert \lesssim q ^{-d/2}, \qquad  \textup{gcd}(a, \ell ,q) =1. 
\end{equation}
In \eqref{e:caq}, the hat indicates the Fourier transform on $ \mathbb Z ^{d}$, and the notation 
identifies the operator $ C _{\lambda } ^{a/q}$, and the kernel. All our operators are convolution operators 
or maximal operators formed from the same.  
The function $ \psi $ is a radial Schwartz function on $ \mathbb R ^{d}$ which satisfies 
\begin{equation}\label{e:psi}
 \mathbf 1_{  \lvert   \xi \rvert  \leq 1/2  } \leq \widetilde \psi  (\xi ) \leq \mathbf 1_{\lvert  \xi \rvert \leq 1 }. 
\end{equation}
The function $ \widetilde  \psi _q (\xi ) = \widetilde \psi (q \xi )$.  
The uniform measure on the sphere of radius $ \lambda $ is denoted by   $ d \sigma _{\lambda }$ and 
$ \widetilde {d \sigma _{\lambda }} $ denotes its Fourier transform computed on $ \mathbb R ^{d}$.  
The standard stationary phase estimate is 
\begin{equation} \label{e:stationary}
 \lvert  \widetilde{d \sigma _1} (\xi ) \rvert \lesssim \lvert  \xi \rvert ^{- \frac{d-1}2}. 
\end{equation} 

We have this estimate, stronger than what we need, from \cite{MSW}*{Prop. 4.1}:  For all $ \Lambda \geq 1$, 
\begin{equation} \label{e:E<}
\bigl \lVert \sup _{\Lambda \leq \lambda \leq 2 \Lambda } \lvert  E_{\lambda }  \cdot  \rvert \bigr\rVert_ {2\to 2 } 
\lesssim \Lambda ^{\frac{4-d}2}. 
\end{equation}
Our first contribution to $ M_2$ is 
$
M _{2,1}  =   \lvert  E_{\tau } f \rvert 
$.
This clearly satisfies the second half of \eqref{e:Ms}. 

It remains to bound 
$
    C _{\tau }  f    
$, requiring further contributions to $ M_1 $ and $ M_2$. 
Recall the estimate below, which is a result of Magyar, Stein and Wainger \cite{MSW}*{Prop. 3.1}. 
\begin{equation}  \label{e:MSWfactor}
\bigl\lVert \sup _{\lambda > q} \lvert  C ^{a/q} _{\lambda }  f  \rvert \bigr\rVert _{2} 
\lesssim q ^{- \frac d2} \lVert f\rVert_2 . 
\end{equation}
It follows that 
\begin{equation}  \label{e:Caq}
\sum_{q > N} \sum _{a \in \mathbb Z _q ^{\times }}  
\lVert    C ^{a/q} _{\tau }  f   \rVert _2 \lesssim N ^{ - \frac{d-4}2} \lVert f\rVert_2. 
\end{equation}
Our second contribution to $ M_2$ is therefore 
\begin{equation} \label{e:M22}
M_ {2,2} = 
 \sum_{ N < q \leq \lambda } \sum _{a \in \mathbb Z _q ^{\times }}  \lvert   C ^{a/q} _{\tau } f \rvert  . 
\end{equation}

We are left with the term below, which will be controlled with further contributions to $ M_1$ and $ M_2$. 
\begin{equation*}
\sum_{1\leq q \leq N} \sum_{a \in \mathbb Z _q ^{\times }} C ^{a/q} _{\tau  }f 
\end{equation*}
Decompose $ C ^{a/q} _{\lambda   } = C ^{a/q} _{\lambda   ,1} + C ^{a/q} _{ \lambda ,2}$ where 
we modify the definition of $c ^{a/q}_{\lambda } $ in \eqref{e:caq} as follows.  
\begin{equation*}
c ^{a/q} _{\lambda   ,1} (\xi ) 
= 
 \sum_{\ell \in \mathbb Z ^{ d}} G(a, \ell, q) \widetilde \psi_ {\lambda /N} (\xi - \ell /q) 
\widetilde {d \sigma _{\lambda }} (\xi - \ell /q) . 
\end{equation*}

The last contribution to $ M _{2}$ is 
\begin{equation*}
M _{2,3} = \Bigl\lvert  \sum_{1\leq q \leq N}  C ^{a/q} _{\tau ,2 } f\Bigr\rvert. 
\end{equation*}
When considering $C ^{a/q} _{\tau ,2 } $, the difference $ \widetilde \psi_ {q } (\xi )-\widetilde \psi_ {\lambda /N} (\xi )$ arises. But this is zero if $ \lvert  \xi \rvert < N/2 \lambda  $.  
Using the Gauss sum estimate \eqref{e:G<} and the stationary decay estimate \eqref{e:stationary}, we have 
by dominating a supremum by an $ \ell ^2 $ sum, 
\begin{align*}
\lVert M _{2,3}\rVert_2 ^2 
&\leq \sum_{k > N} \Bigl\lVert   \sum_{1\leq q \leq N} \sum _{a\in \mathbb Z _q ^{\times }} C ^{a/q} _{\lambda _k ,2 } f \Bigr\rVert_2 ^2 
\\
& \leq  N  \sum_{1\leq q \leq N} \sum_{k > N}\sum _{a\in \mathbb Z _q ^{\times }}q
\lVert     C ^{a/q} _{\lambda _k ,2 } f \rVert_2 ^2 
\\
& \leq N ^{2-d}  \sum_{1\leq q \leq N} q ^{2-d} \lesssim N ^{2-d}. 
\end{align*}
This is smaller than required. 

The principle point is the control of 
\begin{equation*}
M _{1,2, \tau }f =   \sum_{1\leq q \leq N}    \sum _{a\in \mathbb Z _q ^{\times }}  C ^{a/q} _{\tau ,1 } f, 
\end{equation*}
and here we adopt our notation for operators.  In particular, we examine the kernel for the convolution 
operator $ M _{1,2, \lambda  }$.   By a well known computation,   (See \cite{I}*{pg. 1415}, \cite{160904313}*{(42)}, or the detailed argument in \cite{MR4064582}*{Lemma 2.13}.)    
\begin{align}\label{e:mhat}
M _{1,2 ,  \lambda  }  (n) &  = K _{\lambda } (n) \cdot \mathsf C _{N} (\lambda ^2 - \lvert  n\rvert ^2  )  ,
\\
 \label{e:K} 
 \textup{where } \quad K _{\lambda } (n) & 
=  \psi _{\lambda /N} \ast d \sigma _{\lambda } (n)     , 
\\  \label{e:Ram}
\textup{and } \quad
 \mathsf C _{N} (n) 
&= \sum_{1\leq q \leq N} \mathsf c _{q} (n) = 
\sum_{1\leq n \leq N} \sum_{a \in \mathbb Z _{q} ^{\times } } e _{q} (a m).
\end{align}
The terms $ \mathsf c_q $ are  Ramanujan sums,
well-known for having more than square root cancellation.  
We need a further quantification of this fact.   We find this result in a paper by Bourgain  
 \cite{MR1209299}*{(3.43), page 126} 
and will give a short  proof for completeness. (Also see \cite{180906468}.)    We remark that the main result of \cite{MR2869206} gives a precise asymptotic for the expression below for $ j=2$. 
In particular, this result shows that the inequality below is sharp, up the $ \epsilon $ dependence. 

%%%%%%%%%%%%%%%%%%%%%%%%%%%%%% LEMMA LEMMA LEMMA
\begin{lemma}\label{l:Ram} Given $ \epsilon >0$ and integer $j$,   the inequality below holds 
for all integers   $ M > Q ^{j}$.  
\begin{equation} \label{e:Ram<}
  \Biggl[  \frac{1}M \sum_{n \leq M} \Bigl [ \sum_{q\leq Q} \lvert  \mathsf  c _{q} (n) \rvert \Bigr ] ^{j} 
  \Biggr] ^{1/j}
\lesssim   Q ^{1 + \epsilon }   .   
\end{equation}
\end{lemma}
%%%%%%%%%%%%%%%%%%%%%%%%%%%%%% LEMMA LEMMA LEMMA

We postpone the proof of this fact to the end of this section.  We also need 

%%%%%%%%%%%%%%%%%%%%%%%%%%%%%% PROPOSITION PROPOSITION PROPOSITION
\begin{proposition}\label{p:K} For the kernel $ K _{\lambda }$ defined in \eqref{e:K}, we have 
 this maximal inequality, valid 
for any lacunary choice of radii $ \{\lambda _k\}$.  
\begin{equation}\label{e:zN}
\Bigl\lVert  \sup _{k > N }  K _{\lambda_k } 
\ast g \Bigr\rVert _{p} \lesssim  \lVert g\rVert_p, \qquad 1< p <2.  
\end{equation}
\end{proposition}
%%%%%%%%%%%%%%%%%%%%%%%%%%%%%% PROPOSITION PROPOSITION PROPOSITION

%%%%%%%%%%%%%%%%%%%%%%%%%%%%%% PROOF PROOF PROOF
\begin{proof}
This follows by comparison to  lacunary averages on \(  \mathbb R ^{d}  \), which we can do since the inner and outer radii compare favorably, as indicated in Figure~\ref{f:convolve}.  
Let us elaborate.  Consider $ 1 \ll M \ll \lambda $, with $ \lambda /M \gg1$.  
The annulus  $ \textup{Ann} (M , \lambda ) = \{  x \in \mathbb R ^{d}\;:\;   \lvert  \lVert x\rVert - \lambda \rvert < \lambda /M \}$. Then, the volume of the annulus is comparable to $ \lambda ^{d}/M$. And, the number of lattice points is, 
\begin{align*}
\bigl\lvert   \mathbb Z ^{d} \cap \textup{Ann} (M , \lambda ) \bigr\rvert 
&=  \sum_{ \mu  ^2 \in \mathbb N  \;:\;  \lambda (1- 1/M) \leq \mu  \leq \lambda (1+ 1/M)}  
\lvert  \{n \in \mathbb Z ^{d} \;:\; \lvert  n\rvert = \mu  \}\rvert 
\\
& \simeq  \sum_{ \mu  ^2 \in \mathbb N  \;:\;  \lambda (1- 1/M) \leq \mu  \leq \lambda (1+ 1/M)} \mu  ^{d-2}  
\simeq \lvert  \textup{Ann} (M , \lambda ) \rvert.   
\end{align*}
The last equivalence holds as we are summing over approximately $ \lambda ^2 /M$ values of $ \mu $.  
In dimension $ d \geq 5$, we always have a good estimate for the number of lattice points on a sphere. 

\end{proof}
%%%%%%%%%%%%%%%%%%%%%%%%%%%%%% PROOF PROOF PROOF

%%%%%%%%%%%%%%% Figure
\begin{figure}
\begin{tikzpicture}
\draw[gray]  (0,0) circle (3cm); 
\draw[gray]  (0,0) circle (2.65cm); \draw[gray]  (0,0) circle (3.45cm);   
\draw[thick,*->] (0,0) -- (0,3) node[midway,right] {$ \lambda $ }; 
\draw[(-)] (2.65,0) -- (3.45,0) node[midway,above] {$ \lambda /N $};  
%\draw[(-)] (-2.65,0) -- (-3.45,0) node[midway,above] {$ q$};  
%\draw (4, 3.5)  node {$ \lVert d \sigma _{\lambda } \ast \zeta _{(\lambda 2 ^{-v})}\rVert _{\infty } % \lesssim 2 ^{v} \lambda ^{-d}$};  
%\draw (-5, -3)  node {$ \lVert  d \sigma _{\lambda } \ast \psi _{\lambda/N}\rVert _{\infty } % \lesssim  [ q \lambda ^{d-1}] ^{-1} $};  
\end{tikzpicture}
\caption{A sketch to indicate the estimates \eqref{e:zN}. 
The convolution $  d \sigma _{\lambda } \ast \psi _{\lambda/N}$ is essentially supported in an annulus 
around a sphere of radius $ \lambda $ of width about $ \lambda /N$. } 
\label{f:convolve}
\end{figure}
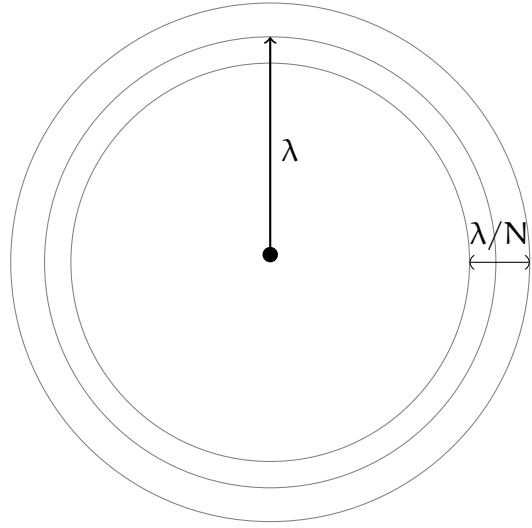
%%%%%%%%%%%%%%% Figure

Let us give the proof of $ \lVert M _{1,2, \tau } f\rVert _{1+ \epsilon } 
\lesssim N ^{1+ \epsilon } \lVert f\rVert _{1+ \epsilon }$, as required for \eqref{e:Ms}. 
We can estimate $ M _{1,2, \tau }f$ from the kernel estimate \eqref{e:mhat}.  We  use H\"older's inequality for a large even integer $ j$, and fixed $ \lambda _k$
\begin{align}
\Bigl\lvert 
\sum_{n\in \mathbb Z ^{d}}  & K_{\lambda }  (n) 
\mathsf C_N (\lambda ^2 - \lvert  n \rvert^2 ) f (x-n) 
\Bigr\rvert
 \\
 & \leq  \bigl[  K _{\lambda} \ast \lvert  f\rvert ^{j'}  (x)  \bigr] ^{1/j'} 
 \times 
 \Bigl[\sum_{n\in \mathbb Z ^{d}} K _{\lambda} (n) 
\lvert \mathsf C_N (\lambda ^2 - \lvert  n \rvert^2 ) \rvert ^{j}  \Bigr] ^{1/j}
\\&:= \Psi _{1, \lambda } f \cdot \Psi _{2, \lambda }.  \label{e:firstTerm} 
\end{align}
We pick $ j \simeq  10/ \epsilon   $, and claim that 
\begin{equation}\label{e:LacClaim}
\bigl\lVert \sup _{k} \Psi _{1, \lambda_k }f\bigr\rVert _{p} \lesssim   \lvert  F\rvert ^{1/p}, 
\qquad \sup _{k > N  } \Psi _{2, \lambda _k} \lesssim N ^{1+ \epsilon }.  
\end{equation}
Indeed,  we have $ 1< j' < p$. Therefore,  we can use \eqref{e:zN} to verify the first claim in \eqref{e:LacClaim}.

Concerning the second term in \eqref{e:firstTerm},  we turn to Lemma~\ref{l:Ram}, and argue that 
\begin{equation*}
 \sup _{k > N  } \Psi _{2, \lambda _k}   \lesssim N ^{ 1 + \epsilon  }
\end{equation*}
from which \eqref{e:LacClaim} follows.  

This follows from a case analysis, depending upon the value of $ n \in \mathbb Z  ^{d}$ in the definition of $ \Psi _{2, \lambda }$. 
For the first case, we sum close to the sphere of radius $ \lambda $.  We have 
\begin{equation}\label{e:CL1}
\sum_{\substack{n\in \mathbb Z ^{d}\\  \lvert  \lambda  -  \lvert  n\rvert \,\rvert \leq N ^{\epsilon } \cdot \lambda /N }} K _{\lambda} (n) 
\lvert \mathsf C_N (\lambda ^2 - \lvert  n \rvert^2 ) \rvert ^{j}   \lesssim N ^{ j (1+ \epsilon) }.  
\end{equation}
Indeed, we have  $ \lVert K _{\lambda } (n)\rVert _{\infty } \lesssim   N / \lambda ^{d}$, and $ K _{\lambda } (n)$ is radial,  and there 
are about $ \lambda ^{d-2}$ lattice points on each sphere.  
The left hand side above is at most 
\begin{align*}  
\frac{N} {\lambda ^2 }\sum_{  \lvert  r - \lambda ^2 \rvert \leq  3\lambda ^2 / N ^{1- \epsilon } }    
\lvert \mathsf C_N (\lambda ^2 - \lvert  n \rvert^2 ) \rvert ^{j}  \lesssim  N ^{ j (1+ 2 \epsilon )} 
\end{align*}
by \eqref{e:Ram<}.  

For the second case, we sum the remaining integer points   $ \lvert  n \rvert <   N   \lambda  $.  There are 
$ N  ^{d}\lambda ^{d}$ such points, and about $ N ^{2 d } \lambda ^2 $ radii that intersect the ball of radius $ N \lambda $.   
For each $ n$ with $ \lvert  \lambda - \lvert  n\rvert\, \rvert  >  N ^{\epsilon } \cdot \lambda /N $, 
note that 
\begin{align*}
K _{\lambda } (n)  \lesssim  \sup _{ \lvert  j\rvert >  N ^{\epsilon } \cdot \lambda /N } \psi _{\lambda /N} (j) 
\lesssim  (N/\lambda ) ^{d}  N ^{- \epsilon J} \lesssim  (N  ^2  \lambda ) ^{-d}, 
\end{align*}
where above we can choose an integer $ J > 4d/\epsilon $ to conclude the estimate, by a Schwartz tails argument. 
It is then clear that the inequality \eqref{e:Ram<} again applies.

The third and final case is $ \lvert  n\rvert >  N \lambda  $.   But this case is very easy, since 
\begin{equation*}
K _{\lambda } (n)  \lesssim   (  N \lambda / \lvert  n\rvert ) ^{-d}    ( N \lvert  n\rvert/ \lambda     ) ^{-J}
\end{equation*} 
Thus,  
\begin{align*}
\sum_{\lvert  n\rvert > N \lambda  }  K _{\lambda } (n) 
& = \sum_{r > (N \lambda ) ^2 }  (N/ \lambda ) ^{d - J}  r ^{- \frac {J-d+2}{2}} 
\\
& \lesssim      (N/ \lambda ) ^{d - J}   (N \lambda ) ^{ -2  \frac{J-d}2}  =  N ^{2d-2J -4}.    
\end{align*}
We can choose $ J$ large.  And, then use the trivial estimate $ \lvert  C_N (n)\rvert  \lesssim N ^2  $ to trivially complete this case.

\begin{proof}[Proof of Lemma~\ref{l:Ram}] 

We will marshal four facts.  
First,  $ n \to \mathsf c _{q} (n)$ is $ q$-periodic, and bounded by $ q$.  Moreover, we have the bound 
$ \lvert  \mathsf c _q (n)\rvert \leq (q,n) $.  To see this, recall that if $ q$ is a power of a prime $ p$, we have 
\begin{equation*}
\mathsf c _{p ^{k}} (n)  = \begin{cases}
0   &   p ^{k-1} \nmid n 
\\
- p ^{k-1}  & p ^{k-1} \mid n , p ^{k} \nmid n 
\\
p ^{k} (1-1/p)  & p ^{k} \mid n 
\end{cases}
\end{equation*}
We see that the conclusion holds in this case.  The general case follows since $ \mathsf c _{q} (n)$ is multiplicative in $ q$. 

\smallskip 

Second, for $ \vec q = (q_1 ,\dotsc, q_j) \in [1,Q] ^{k}$, let $ \mathcal L (\vec q)$ be the least common multiple of $ q_1 ,\dotsc, q_k$.  Observe that $ n \to  \prod _{j=} ^{k}\mathsf c _{q_j} (n)$
is periodic with period $ \mathcal L (\vec q)$. This, with the condition that $ M > Q ^j $,  implies that 
\begin{align}  \label{e:R3}
\frac{1}M \sum_{n \leq M}    \prod _{i=1} ^{j}\lvert  \mathsf c _{q_j} (n)\rvert  
\leq \frac{2 } {\mathcal L (\vec q)} \sum_{n\leq \mathcal L (\vec q)} \prod _{i=1} ^{j} (q_j,n).  
\end{align}

\smallskip 
Third,  for all $ \epsilon >0$, uniformly in $ \vec q \in  [1,Q] ^{k}$, 
\begin{equation}\label{e:R2}
 \sum_{n\leq \mathcal L (\vec q)} \prod _{i=1} ^{k} (q_j,n) \lesssim Q ^{k +  \epsilon }. 
\end{equation}
To see this, begin with the case of $ q = p ^{x}$, for prime $ p$ and $ x \geq 1$.  For integers $ k$, 
\begin{equation*}
\sum_{ n \leq p ^{x}} (p ^{x} ,n) ^{k} \lesssim p ^{xk + \epsilon }, 
\end{equation*}
as is easy to check.  
We need an extension of this.  Let $ x_1 ,\dotsc, x_t$ be distinct integers, and let $ k_1 ,\dotsc, k_t$ be integers.  There holds 
\begin{equation}  \label{e:RR1}
\sum_{ n \leq p ^{x_1}} \prod _{s=1} ^{t} (p ^{x_s} ,n) ^{k} \lesssim p ^{  \sum_{s=1} ^{t}x_s k_s + \epsilon }. 
\end{equation}
where above we assume that $ x_1 > x_2 > \cdots > x_t$.  As $ n \to (p ^{x_s} ,n) ^{k}$ is periodic with period $ p ^{x_s}$, one has 
\begin{equation*}
\sum_{ n \leq p ^{x}} \prod _{s=1} ^{t} (p ^{x_s} ,n) ^{k} 
= \prod _{s=1} ^{t}\sum_{ n \leq p ^{x_s}}  (p ^{x_s} ,n) ^{k} 
\end{equation*}
and the claim follows.

Turning to a vector $ \vec q$, write the prime factorization of $ \mathcal L (\vec q) = p_1 ^{x_1} \cdots p _t^{x_t}$.  Write each $ q_j = \prod _{s=1} ^{t} p _{s} ^{y_s}$, where $ 0\leq y_s \leq x_s$.  Then, 
for appropriate integers $ k _{y}$, we have 
\begin{equation*}
 \prod _{i=1} ^{k} (q_j,n)  
 =  \prod _{s=1} ^{t} \prod _{y=1} ^{x_s} (p_s ^{y} ,n) ^{k_y} .  
\end{equation*}
One must note that $  \prod _{y=1} ^{x_s} (p_s ^{y}) ^{k_y} \leq Q ^{k}$.  Again appealing to periodicity and using \eqref{e:RR1}, we can then write 
\begin{align*}
 \sum_{n\leq \mathcal L (\vec q)} \prod _{i=1} ^{k} (q_j,n)  
 & = \sum_{n\leq \mathcal L (\vec q)}  \prod _{s=1} ^{t} \prod _{y=1} ^{x_s} (p_s ^{y} ,n) ^{k_y} 
 \\
 & =  \prod _{s=1} ^{t} 
 \sum_{n\leq  p ^{x_s}}    \prod _{y=1} ^{x_s} (p_s ^{y} ,n) ^{k_y} 
 \lesssim  \prod _{s=1} ^{t}  p_s ^{\epsilon + \sum _{y=1} ^{x_s} y \cdot k_y } \lesssim Q ^{\epsilon +k}. 
\end{align*}

\smallskip
Fourth, we have the inequality below, valid for all $ \epsilon >0$
\begin{equation}\label{e:R4}
\sum_{ \vec q \in [1,Q]^j} \frac{1} {\mathcal L (\vec q)} \lesssim Q ^{\epsilon }. 
\end{equation}
Appealing to the divisor function $ d(r) = \sum _{q\leq r : q|r} 1  $, and the estimate 
$ d (r) \lesssim r ^{\epsilon }$, we have 
\begin{align*}
\sum_{ \vec q \in [1,Q]^j} \frac{1} {\mathcal L (\vec q)}   
& \leq \sum_{q \leq Q ^{j}}   \frac{d (q) ^{j}}q    \lesssim Q ^{\epsilon j}.  
\end{align*}
As $ \epsilon >0$ is arbitrary, we are finished.  

\smallskip 

We turn to the main line of the argument.  Estimate 
\begin{align*}
 \frac{1}M \sum_{n \leq M} \Bigl [ \sum_{  q\leq Q} \lvert  c _{q} (n) \rvert  \Bigr] ^{j}
 & = \frac 1 M 
 \sum_{n\leq M} \sum_{ \vec q \in [1,Q]^j}
  \prod _{i=} ^{j}\lvert  \mathsf c _{q_j} (n)\rvert  
  \\
&\stackrel {\eqref{e:R3}}\lesssim   
 \sum_{ \vec q \in [1,Q]^j}  \sum_{n \leq \mathcal L (\vec q)} 
  \prod _{i=} ^{j} (q_j ,n) 
  \\
  &\stackrel {\eqref{e:R2}}\lesssim   
 \sum_{ \vec q \in [1,Q]^j}  \frac{Q ^{\epsilon +j}} {\mathcal L (\vec q)} 
  \stackrel{\eqref{e:R4}}\lesssim Q ^{2\epsilon +j}. 
\end{align*}
This is our bound \eqref{e:Ram<}.  

\end{proof}
%%%%%%%%%%%%%%%%%%%%%%%%%%%%%% PROOF PROOF PROOF

%%%%%%%%%%%%%%%%%%%%%%%%%%%%%% SECTION  SECTION SECTION
%%%%%%%%%%%%%%%%%%%%%%%%%%%%%% SECTION  SECTION SECTION 
\section{The Highly Composite Case} \label{s:HC}

We follow the lines of the previous argument, but the underlying details are substantially different, 
as we are modifying Cook's argument \cite{180803822}, also see \cite{171104298C}.  
The essential features are due to Cook. We hope that this way of presenting the proof makes the argument more accessible.  

The  point is to show that for any $ 0< \epsilon  <1$, and  $ f = \mathbf 1_{F}$, a finitely supported function,  stopping time  $ \tau \;:\; \mathbb Z ^{d} \to \{\lambda _k\} $, 
and  any integer $ N$,  
we can choose $ M_1$ and $ M_2$ so that 
$ A _{\tau  } f \leq M_1 + M_2$ where 
\begin{align}\label{e:HCpre1}
\lVert M_1\rVert_p & \lesssim N ^{\epsilon } \lvert F\rvert ^{1/p}, 
\\ \label{e:HCpre2}
\lVert M_2\rVert_2  & \lesssim N ^{ - \frac{4-d}2 } \lvert F\rvert ^{1/2}. 
\end{align}
The implied constants depend upon $ \epsilon >0$.  This proves our Theorem~\ref{t:HC}.   
In the statement of this Theorem, recall that $ \log \mu _u / \log k \to \infty $, and that $ \lambda _k ^2  = \mu _k !$.

Fix $ \epsilon >0$. It suffices to prove \eqref{e:HCpre1} and \eqref{e:HCpre2} for sufficiently large $ N>N_0$.  Recall that $ \lambda _k ^2 = \mu _k !$. By our key assumption \eqref{e:log},  namely that $ \mu _k$ grows faster than any polynomial, there is a choice of $ N_0$ 
so that for all $ N>N_0$, we have $ \lambda  _{ [N ^{\epsilon }]}  > N ^{3 } $.  
For these integers, the first contribution to $ M_1$ is $ M _{1,1} f =  \mathbf 1_{\tau \leq \lambda _{[N ^{\epsilon }]}}  A _{\tau } f$. This clearly satisfies \eqref{e:HCpre1}.    
We can assume that $ \tau > \lambda _{[N ^{\epsilon }]}$ below.

The decomposition of the averages $ A _{\lambda_k }$ is different from that in \eqref{e:Afull}.  
Modify the definition in \eqref{e:caq} as follows.  Set $ Q= N!$, and define 
\begin{align}\label{e:B}
b _{\lambda } (\xi ) = \sum_{0\leq a <Q} \sum_{\ell \in \mathbb Z _Q ^{d}}
G (a, \ell, Q ) \widetilde \psi_{2Q} (\xi - \ell /Q) 
\widetilde {d \sigma _{\lambda }} (\xi - \ell /Q) .  
\end{align}
Note that this is a very big sum. In particular it is typical to restrict Gauss sums $ G (a, \ell, Q )$ to the case 
where $ \textup{gcd}(a, \ell ,Q)=1$, but we are not doing this here. 
Our first contribution to $ M_2$ is $ M _{2,1}f  =   \lvert  B _{\tau } f - A _{ \tau } f\rvert $.  Here, we are adopting our conventions about operators and their multipliers.  

%%%%%%%%%%%%%%%%%%%%%%%%%%%%%% LEMMA LEMMA LEMMA
\begin{lemma}\label{l:m22} 
We have the estimate $ \lVert M _{2,1} f\rVert_2 \lesssim N^ { \frac {4-d}2}  \lvert  F\rvert ^{1/2} $. 
\end{lemma}
%%%%%%%%%%%%%%%%%%%%%%%%%%%%%% LEMMA LEMMA LEMMA

%%%%%%%%%%%%%%%%%%%%%%%%%%%%%% PROOF PROOF PROOF
\begin{proof}
The difference $ M _{2,1} f$ is split into several terms. Using the expansion of $ A _{\lambda }$ from \eqref{e:Afull}, the expansion is 
\begin{align} 
 M _{2,1}f  & \leq    \lvert  E _{ \tau  }f\rvert 
 +  \sum_{ q >N}  \sum_{ a \in \mathbb Z _q ^{\times }}    \lvert  C ^{a/q} _{ \tau } f  \rvert 
 \\  \label{e:Diff1}
 & \qquad   +   \Bigl\lvert B _{\tau  } f  - \sum_{1\leq q \leq N} \sum_{a \in \mathbb Z _q ^{\times }}  e_q (- \tau ^2  a) C ^{a/q} _{\tau } f \Bigr\rvert . 
\end{align}
We bound the $ \ell ^2 $ norm of each of these terms in order.   

The first term on the right is bounded by appeal to \eqref{e:E<}.  The second term on the right is bounded by appeal to \eqref{e:Caq}.  Thus, it is the third term  \eqref{e:Diff1} that is crucial.  
We have this critical point about the term $e_q (- \tau ^2  a)$ appearing in \eqref{e:Diff1}.  
The stopping time $\tau$ takes values in $ \{\lambda_k : k > N^\epsilon\}$.  
  The highly composite nature of the $ \lambda _k$ shows that $ e_q (- \lambda ^2 _k a) \equiv 1$, 
for $ k >N ^{\epsilon }$, $ 1\leq q \leq N$, and $ a \in \mathbb Z _q ^{\times }$.  
(Indeed, this is the crucial simplifying feature of the highly composite case.) 
And so the term in \eqref{e:Diff1} is 
\begin{align*}
 B _{\tau  } f  - \sum_{1\leq q \leq N} \sum_{a \in \mathbb Z _q ^{\times }}   C ^{a/q} _{\tau } f . 
\end{align*}
For a fixed value of $ \tau $, the multiplier above is 
\begin{align} \label{e:DD}
\begin{split}
\sum_{0\leq  a' <Q} \sum_{\ell' \in \mathbb Z _Q ^{d}} &
G (a', \ell ', Q ) \widetilde \psi_{2Q} (\xi - \ell '/Q) 
\widetilde {d \sigma _{\lambda }} (\xi - \ell' /Q)  
\\&- \sum_{1\leq q \leq N} \sum_{a \in \mathbb Z _q ^{\times }} \sum_{\ell \in \mathbb Z ^{d}_q}
G (a, \ell ,q)  \widetilde \psi_q (\xi - \ell /q) 
\widetilde {d \sigma _{\lambda }} (\xi - \ell/q) . 
\end{split}
\end{align}

We need to closely track cancellation in the difference above. 
Recall the following basic property of Gauss sums.  For $ a', \ell' ,Q$ as above, 
we have 
\begin{equation}\label{e:G=}
G (a', \ell' , Q) = 
\begin{cases}
G (a'/ \rho , \ell '/ \rho , Q/\rho )  & \rho = \rho _{a', \ell'} = \textup{gcd} (a', \ell ', Q), 
\\
0   &  \textup{gcd} (a', Q) > \rho . 
\end{cases}
\end{equation}
It follows that the difference \eqref{e:DD} splits naturally between the two cases 
when  for fixed $a', \ell' $ we have $ Q/ \rho $ being either strictly bigger than $ N$ or less than or equal to $ N$. 

In the case of $ Q/ \rho \leq N$,   and $ \textup{gcd} (a',Q) = \rho $,  we have cancellation between the two terms.  Define 
\begin{equation*}
t _{a', \lambda } (\xi ) = 
 \sum_{ \substack {\ell' \in \mathbb Z _Q^{d } \\  Q/\rho _{a', \ell'}  \leq N}}
  G (a ', \ell ' ,Q)  \{ \widetilde \psi_{2Q}  (\xi - \ell '  /Q) - \widetilde \psi_{Q/\rho}  (\xi - \ell' /Q)  \} 
\widetilde {d \sigma _{\lambda }} (\xi - \ell/q). 
\end{equation*}
Notice that the difference 
$
 \{ \widetilde \psi_{2Q} (\xi  ) - \widetilde \psi_{Q/\rho}  (\xi )  \} 
$ 
is zero for $ \lvert \xi \rvert < (4Q)^{-1}$.  
We have by a square function argument and the stationary phase estimate \eqref{e:stationary}, 
\begin{align*}
  \sum _{k > N ^{\epsilon }} \lVert   T _{a', \lambda _k } f  \rVert_2 ^2 
   \lesssim \lVert f\rVert_2 ^2  \sum _{k > N ^{\epsilon }} (Q/ \lambda _k) ^{1-d} \lesssim Q ^{2 (1-d)} \lvert  F\rvert,  
\end{align*}
since we have $ \mu _{[N ^{\epsilon }]} > N ^{3}$, and so $ \lambda _k ^2 = \mu _k !  \geq N ^{3}!$, while $ Q= N!$.    This is summed over $0\leq a' < Q$ to give a smaller 
estimate than claimed. 

In the case of $ Q/ \rho > N$,   only the first half of \eqref{e:G=} is non-zero.  
Modification of the argument that leads to \eqref{e:Caq} will complete the proof.  Fix $ q>N$, and set 
\begin{equation*}
s _{\lambda } (\xi )
= \sum_{a' \in \mathbb Z _Q ^{d} }
\sum_{ \substack {\ell' \in \mathbb Z _Q^{d } \\  Q/\rho _{a', \ell'}  =q}}
  G (a ', \ell ' ,Q)   \widetilde \psi_{2Q}  (\xi - \ell '  /Q)  
\widetilde {d \sigma _{\lambda }} (\xi - \ell/q). 
\end{equation*}
This differs from $ \sum_{a \in \mathbb Z _q } c ^{a/q} _{\tau } $ by only the cut-off term $ \widetilde \psi_{2Q}  ( \cdot ) $.  This is however a trivial term, due to our growth condition on $ \lambda _ k$
and the stationary decay estimate \eqref{e:stationary}.  
Note that from the Gauss sum estimate \eqref{e:G=}, and an easy square function argument, 
and \eqref{e:G<},  we have 
\begin{equation*}
\lVert S _{\tau } f\rVert_2 \lesssim 
q ^{1- \frac{d}2}\lVert f\rVert_2 + 
\sum_{a \in \mathbb Z _q } \lVert C ^{a/q} _{\tau } f\rVert_2 . 
\end{equation*}
But then, we can complete the proof from \eqref{e:Caq}.   And the proof is finished.
\end{proof}
%%%%%%%%%%%%%%%%%%%%%%%%%%%%%% PROOF PROOF PROOF

It remains to  consider $ M _{1,2} f =   \mathbf 1_{\tau > \lambda _{[N ^{\epsilon }]}} \lvert  B _{\tau  } f\rvert $, where $ B _{\lambda } f$ 
is defined in \eqref{e:B}. We show that it satisfies the $\ell^p$ estimate  \eqref{e:HCpre1}, using a variant of the 
factorization argument of Magyar, Stein and Wainger \cite{MSW}.  The factorization is given by  $ B _{\lambda } = T _{\lambda } \circ U $, where the multipliers for these operators are given by 
\begin{align}
t _{\lambda } (\xi ) & = 
\sum_{0\leq a <Q} \sum_{\ell \in \mathbb Z _Q ^{d}}
 \widetilde \psi_{2Q} (\xi - \ell /Q) 
\widetilde {d \sigma _{\lambda }} (\xi - \ell /Q) , 
\\  \label{e:u}
\textup{and} \qquad u (\xi ) &= 
\sum_{0\leq a <Q} \sum_{\ell \in \mathbb Z _Q ^{d}}
G (a, \ell, Q ) \widetilde \psi_{Q} (\xi - \ell /Q) . 
\end{align}
Namely, the multiplier $ t _{\lambda }$ is $ 1/Q$-periodic in each coordinate, and has the spherical part of the multiplier. 
All the Gauss sum terms are in $ u (\xi )$.  
The fact that $ B _{\lambda }= T _{\lambda } \circ U $ follows from choice of $ \psi $ in \eqref{e:psi}.  

Concerning the maximal operator $    T _{\tau } \phi  $, we can appeal to the transference result of 
\cite{MSW}*{Prop 2.1} to bound $\ell ^{p}$ norms of this maximal operator.  
Since the lacunary spherical maximal function is bounded on all $ L ^{p} (\mathbb R ^{d})$, we conclude 
that 
\begin{equation*}
\lVert  T _{\tau } \phi  \rVert_ {\ell ^{p}} 
\lesssim \lVert \phi \rVert _{\ell ^{p}}, \qquad 1< p < \infty .  
\end{equation*}
Apply this with $ \phi = U f $.
It remains to see that $ U f $ is bounded in the same range.  But this is the proposition below, 
which concludes the proof of \eqref{e:HCpre1}, and hence the proof of Theorem~\ref{t:HC}. 

%%%%%%%%%%%%%%%%%%%%%%%%%%%%%% PROPOSITION PROPOSITION PROPOSITION
\begin{proposition}\label{p:U} For $ 1\leq p \leq 2$, we have $ \lVert U f \rVert_p \lesssim \lVert f\rVert_p$. 

\end{proposition}
%%%%%%%%%%%%%%%%%%%%%%%%%%%%%% PROPOSITION PROPOSITION PROPOSITION

%%%%%%%%%%%%%%%%%%%%%%%%%%%%%% PROOF PROOF PROOF
\begin{proof}
The $ \ell  ^2 $ estimate follows Plancherel and $ \lVert u\rVert _{\infty } \lesssim 1$.  It remains to verify the $ \ell ^{1}$ estimate. But, that amounts to the estimate $ \lVert U\rVert_1 = \sum_{m} \lvert  U (m)\rvert  \lesssim 1$. And so we compute 
\begin{align*}
U (-m) & = \int _{\mathbb T ^{d}} u (\xi ) e ^{-i m \cdot \xi } \; d \xi  
\\
&= \sum_{0\leq a < Q}  
\sum_{\ell \in \mathbb Z _Q ^{d}}
G (a, \ell, Q ) \int _{\mathbb T ^{d}}   \widetilde \psi_{Q} (\xi - \ell /Q)e ^{-i m \cdot \xi } \; d \xi 
\\
& = \psi _Q (m)
 \sum_{0\leq a < Q}  
\sum_{\ell \in \mathbb Z _Q ^{d}}
G (a, \ell, Q ) e ^{-i m  \cdot \ell /Q} 
\\
& = \frac { \psi _Q (m)}{Q ^{d}}
 \sum_{0\leq a < Q}   \sum_{n \in \mathbb Z _Q ^{d}}
\sum_{\ell \in \mathbb Z _Q ^{d}}   e _Q ( a \lvert  n\rvert ^2 +  (n-m) \cdot  \ell ) 
\\
& =   \frac { \psi _Q (m)}{Q ^{d-1}}    \sum_{n \in \mathbb Z _Q ^{d}}   \sum_{\ell \in \mathbb Z _Q ^{d}}
 e _Q (   (n-m) \cdot  \ell  )  \delta _{ \{ \lvert  n\rvert ^2  \equiv 0 \mod Q \} }
 \\
 & = Q \psi _Q (m)  \delta _{ \{ \lvert  m \rvert ^2 \equiv 0 \mod Q \} } . 
\end{align*}

And, then, recalling \eqref{e:psi}, it follows that 
\begin{align*}
\lVert U\rVert_1 = \sum_{m} \lvert  U (m)\rvert & \lesssim   Q ^{1-d}\sum_{ \lvert  m\rvert \leq Q}  
 \delta _{ \{ \lvert  m \rvert ^2 \equiv 0 \mod Q \} } 
 \\
 & \lesssim Q ^{1-d} \sum_{j =1} ^{ Q} \lvert  j Q\rvert ^{\frac{d-2}2} 
 \lesssim Q ^{-d/2} \sum_{j =1} ^{ Q} j ^{ \frac{d}2 -1} \lesssim 1.  
\end{align*}

\end{proof}
%%%%%%%%%%%%%%%%%%%%%%%%%%%%%% PROOF PROOF PROOF

A more general version of this last Lemma is proved in \cite{171104298C}*{Lemma 15}. 

%\bibliography{radon,sparse} 

\begin{bibdiv}
\begin{biblist}

\bib{MR1209299}{article}{
      author={Bourgain, J.},
       title={Fourier transform restriction phenomena for certain lattice
  subsets and applications to nonlinear evolution equations. {I}.
  {S}chr\"odinger equations},
        date={1993},
        ISSN={1016-443X},
     journal={Geom. Funct. Anal.},
      volume={3},
      number={2},
       pages={107\ndash 156},
         url={https://doi-org.prx.library.gatech.edu/10.1007/BF01896020},
      review={\MR{1209299}},
}

\bib{MR812567}{article}{
      author={Bourgain, Jean},
       title={Estimations de certaines fonctions maximales},
        date={1985},
        ISSN={0249-6291},
     journal={C. R. Acad. Sci. Paris S\'er. I Math.},
      volume={301},
      number={10},
       pages={499\ndash 502},
      review={\MR{812567}},
}

\bib{MR2869206}{article}{
      author={Chan, T.~H.},
      author={Kumchev, A.~V.},
       title={On sums of {R}amanujan sums},
        date={2012},
        ISSN={0065-1036},
     journal={Acta Arith.},
      volume={152},
      number={1},
       pages={1\ndash 10},
         url={https://doi-org.prx.library.gatech.edu/10.4064/aa152-1-1},
      review={\MR{2869206}},
}

\bib{2017arXiv170301508C}{article}{
      author={{Cladek}, L.},
      author={{Krause}, B.},
       title={{Improved endpoint bounds for the lacunary spherical maximal
  operator}},
        date={2017-03},
     journal={ArXiv e-prints},
      eprint={1703.01508},
}

\bib{171104298C}{article}{
      author={{Cook}, B.},
       title={{Maximal Function Inequalities and a Theorem of Birch}},
        date={2017-11},
     journal={ArXiv e-prints},
      eprint={1711.04298},
}

\bib{180803822}{article}{
      author={{Cook}, B.},
       title={{A note on discrete spherical averages over sparse sequences}},
        date={2018},
     journal={ArXiv e-prints},
      eprint={1808.03822},
}

\bib{160904313}{article}{
      author={{Hughes}, K.},
       title={{The discrete spherical averages over a family of sparse
  sequences}},
        date={2016-09},
     journal={ArXiv e-prints},
      eprint={1609.04313},
}

\bib{I}{article}{
      author={Ionescu, Alexandru~D.},
       title={An endpoint estimate for the discrete spherical maximal
  function},
        date={2004},
        ISSN={0002-9939},
     journal={Proc. Amer. Math. Soc.},
      volume={132},
      number={5},
       pages={1411\ndash 1417},
  url={https://doi-org.prx.library.gatech.edu/10.1090/S0002-9939-03-07207-1},
      review={\MR{2053347}},
}

\bib{180906468}{article}{
      author={{Kesler}, R.},
       title={{$\ell^p(\mathbb{Z}^d)$-Improving Properties and Sparse Bounds
  for Discrete Spherical Maximal Means, Revisited}},
        date={2018-09},
     journal={ArXiv e-prints},
      eprint={1809.06468},
}

\bib{MR4064582}{article}{
   author={Kesler, R.},
   author={Lacey, M. T.},
   title={$\ell^p$-Improving Inequalities for Discrete Spherical Averages},
   journal={Anal. Math.},
   volume={46},
   date={2020},
   number={1},
   pages={85--95},
   issn={0133-3852},
   review={\MR{4064582}},
   doi={10.1007/s10476-020-0019-9},
}

\bib{MR4041278}{article}{
   author={Kesler, Robert},
   author={Lacey, Michael T.},
   author={Mena, Dar\'{\i}o},
   title={Sparse bounds for the discrete spherical maximal functions},
   journal={Pure Appl. Anal.},
   volume={2},
   date={2020},
   number={1},
   pages={75--92},
   issn={2578-5885},
   review={\MR{4041278}},
   doi={10.2140/paa.2020.2.75},
}

\bib{MR4041115}{article}{
   author={Lacey, Michael T.},
   title={Sparse bounds for spherical maximal functions},
   journal={J. Anal. Math.},
   volume={139},
   date={2019},
   number={2},
   pages={613--635},
   issn={0021-7670},
   review={\MR{4041115}},
   doi={10.1007/s11854-019-0070-2},
}


\bib{MR1949873}{article}{
      author={Lee, Sanghyuk},
       title={Endpoint estimates for the circular maximal function},
        date={2003},
        ISSN={0002-9939},
     journal={Proc. Amer. Math. Soc.},
      volume={131},
      number={5},
       pages={1433\ndash 1442},
  url={http://dx.doi.org.prx.library.gatech.edu/10.1090/S0002-9939-02-06781-3},
      review={\MR{1949873}},
}

\bib{MR0358443}{article}{
      author={Littman, Walter},
       title={{$L\sp{p}-L\sp{q}$}-estimates for singular integral operators
  arising from hyperbolic equations},
        date={1973},
       pages={479\ndash 481},
      review={\MR{0358443}},
}

\bib{MSW}{article}{
      author={Magyar, A.},
      author={Stein, E.~M.},
      author={Wainger, S.},
       title={Discrete analogues in harmonic analysis: spherical averages},
        date={2002},
        ISSN={0003-486X},
     journal={Ann. of Math. (2)},
      volume={155},
      number={1},
       pages={189\ndash 208},
         url={https://doi-org.prx.library.gatech.edu/10.2307/3062154},
      review={\MR{1888798}},
}

\bib{MR1617657}{article}{
      author={Magyar, Akos},
       title={{$L^p$}-bounds for spherical maximal operators on {$ \mathbf{Z}^n$}},
        date={1997},
        ISSN={0213-2230},
     journal={Rev. Mat. Iberoamericana},
      volume={13},
      number={2},
       pages={307\ndash 317},
         url={https://doi-org.prx.library.gatech.edu/10.4171/RMI/222},
      review={\MR{1617657}},
}

\bib{MR1388870}{article}{
      author={Schlag, W.},
       title={A generalization of {B}ourgain's circular maximal theorem},
        date={1997},
        ISSN={0894-0347},
     journal={J. Amer. Math. Soc.},
      volume={10},
      number={1},
       pages={103\ndash 122},
  url={http://dx.doi.org.prx.library.gatech.edu/10.1090/S0894-0347-97-00217-8},
      review={\MR{1388870}},
}

\bib{MR1955209}{article}{
      author={Seeger, Andreas},
      author={Tao, Terence},
      author={Wright, James},
       title={Endpoint mapping properties of spherical maximal operators},
        date={2003},
        ISSN={1474-7480},
     journal={J. Inst. Math. Jussieu},
      volume={2},
      number={1},
       pages={109\ndash 144},
  url={http://dx.doi.org.prx.library.gatech.edu/10.1017/S1474748003000057},
      review={\MR{1955209}},
}

\bib{MR0256219}{article}{
      author={Strichartz, Robert~S.},
       title={Convolutions with kernels having singularities on a sphere},
        date={1970},
        ISSN={0002-9947},
     journal={Trans. Amer. Math. Soc.},
      volume={148},
       pages={461\ndash 471},
         url={http://dx.doi.org.prx.library.gatech.edu/10.2307/1995383},
      review={\MR{0256219}},
}

\bib{JZ}{article}{
      author={Zienkiewicz, J.},
       title={Personal communication},
}

\end{biblist}
\end{bibdiv}

\bibliographystyle{alpha,amsplain}	

\end{document}